\documentclass[10pt,reqno]{amsart}

\usepackage{a4wide}
\usepackage{dsfont}          
\usepackage{amssymb,amsmath}
\usepackage{mathrsfs}
\usepackage[all]{xy}

\newtheorem{proposition}{Proposition}[section]
  \newtheorem{theorem}[proposition]{Theorem}
  
  \newtheorem{lemma}[proposition]{Lemma}
\theoremstyle{remark}
  \newtheorem{definition}[proposition]{Definition}
  \newtheorem{remark}[proposition]{Remark}

\newcommand{\cst}{\ifmmode\mathrm{C}^*\else{$\mathrm{C}^*$}\fi}
\newcommand{\st}{\;\vline\;}
\newcommand{\CC}{\mathbb{C}}
\newcommand{\RR}{\mathbb{R}}
\newcommand{\TT}{\mathbb{T}}

\newcommand{\tens}{\otimes}
\newcommand{\vtens}{\,\bar{\otimes}\,}
\newcommand{\id}{\mathrm{id}}
\newcommand{\comp}{\circ}
\renewcommand{\th}{\vartheta}
\newcommand{\rh}{\varrho}
\newcommand{\thL}{\vartheta^{\text{\tiny\rm{L}}}\;\!\!}
\newcommand{\hhthL}{\hh{\vartheta}^{\text{\tiny\rm{L}}}\;\!\!}
\newcommand{\I}{\mathds{1}}
\newcommand{\GG}{\mathbb{G}}
\newcommand{\HH}{\mathbb{H}}
\newcommand{\sB}{\mathsf{B}}
\newcommand{\sM}{\mathsf{M}}
\newcommand{\sA}{\mathsf{A}}
\newcommand{\sN}{\mathsf{N}}
\newcommand{\sK}{\mathsf{K}}
\newcommand{\sL}{\mathsf{L}}
\newcommand{\hh}[1]{\widehat{#1}}
\newcommand{\dd}[1]{\widetilde{#1}}
\newcommand{\op}{\text{\rm\tiny{op}}}
\newcommand{\flip}{\boldsymbol{\sigma}}

\DeclareMathOperator{\C}{C}
\DeclareMathOperator{\B}{B}
\DeclareMathOperator{\Mor}{Mor}
\DeclareMathOperator{\linW}{\overline{span}^{\:\!\text{\tiny\rm{w}}}\:\!\!}
\DeclareMathOperator{\Linf}{\mathnormal{L}^\infty\;\!\!}
\DeclareMathOperator{\Ltwo}{\mathnormal{L}^2\;\!\!}

\numberwithin{equation}{section}

\author{Pawe{\l} Kasprzak}
\address{Department of Mathematical Methods in Physics, Faculty of Physics, University of Warsaw, Poland}
\email{pawel.kasprzak@fuw.edu.pl}

\author{Piotr M.~So{\l}tan} 
\address{Department of Mathematical Methods in Physics, Faculty of Physics, University of Warsaw, Poland}
\email{piotr.soltan@fuw.edu.pl}

\thanks{Supported by National Science Centre (NCN) grant no.~2015/17/B/ST1/00085}

\title[Quantum groups with projection and extensions]{Quantum groups with projection and extensions of locally compact quantum groups}



\begin{document}

\begin{abstract}
The main result of the paper is the characterization of those locally compact quantum groups with projection, i.e.~non-commutative analogs of semidirect products, which are extensions as defined by L.~Vainerman and S.~Vaes. It turns out that quantum groups with projection are usually not extensions. We discuss several examples including the quantum $\mathrm{U}_q(2)$. The major tool used to obtain these results is the co-duality for coideals in algebras of functions on locally compact quantum groups and the concept of a normal coideal.
\end{abstract}

\maketitle

\section{Introduction}\label{intro}

This is our second paper devoted to studying analogs of semidirect products in the theory of locally compact quantum groups. In the previous paper \cite{proj}, motivated by the pioneering work of \cite{RoyPhd} (with roots in \cite{PodSLW1,PodSLW2}) we introduced the notion of a locally compact quantum group with projection (onto another locally compact quantum group) on von Neumann algebra level, cf.~Section \ref{projSect}. Classically this structure corresponds to that of a semidirect product of locally compact groups. It turns out that in the non-commutative (or \emph{quantum}) setting the object which classically corresponds to the normal subgroup in the semidirect product decomposition of a group might fail to be a quantum group. The relevant structure is that of a \emph{braided quantum group} (cf.~\cite{WorBraid,RoyPhd,kmrw,proj}). The question we want to address in this paper concerns the relationship of the notion of a quantum group with projection with extensions of locally compact quantum groups defined in \cite{VaesVainerman}. The precise formulation of this question is given in Section \ref{ExtSect}, where in Theorem \ref{main} we provide a complete characterization of quantum groups with projection which are extensions. We provide examples which show that typically a quantum group with projection is not an extension.

The key tool we use in our analysis is the co-duality for coideals in algebras of functions on locally compact quantum groups (introduced in \cite{KaspSol}) and  the notions of a normal and strongly normal coideals. The rough picture of the situation is that even in cases where the ``normal subgroup'' required by the definition of an extension is missing in a locally compact quantum group with projection, the coideal which in case of an extension would correspond to the normal subgroup via co-duality still exists and has a property which we call \emph{normality} (cf.~Section \ref{normalSect}).

The paper is organized as follows: in Section \ref{Prel} we collect information concerning the basic set-up of the theory of locally compact quantum groups and recall the standard notation and terminology. Section \ref{coduality} contains an exposition of results about co-duality for coideals and several simple results concerning invariance of coideals under the unitary antipode and the scaling group and discussion of coideals related to closed quantum subgroups. Then, in Section \ref{normalSect}, the notions of  normal and strongly normal coideals are introduced and discussed. Section \ref{projSect} recalls the results of \cite{proj} and Section \ref{ExtSect} poses and answers the main question of the paper. Finally, in Section \ref{examplesSect}, a wealth of examples of locally compact quantum groups with projection are provided and the question whether they are extensions is settled. 

\section{Preliminaries}\label{Prel}

\subsection{Locally compact quantum groups}\label{lcqgs}

The paper deals with theory of locally compact quantum groups developed primarily in the seminal paper \cite{kv}. We will use the notation already adopted by several authors and write symbols like $\GG$ and $\HH$ to denote locally compact quantum groups and use $\Linf(\GG)$ and $\C_0(\GG)$ for the von Neumann algebra and \cst-algebra associated with $\GG$ (see e.g.~\cite{DKSS}). We refer to these algebras as algebras of functions on $\GG$. Both these algebras are equipped with comultiplications (we have $\Linf(\GG)=\C_0(\GG)''$ and the comultiplication on $\Linf(\GG)$ is an extension of that on $\C_0(\GG)$), and we will denote them both by the same symbol $\Delta_\GG$. For the theory of \cst-algebras and von Neumann algebras and its use in the theory of quantum groups we refer the reader e.g.~to \cite{kv,mnw,mu}. The symbol ``$\vtens$'' will denote the tensor product of von Neumann algebras while ``$\tens$'' will denote either the spatial tensor product of \cst-algebras, tensor product of Hilbert spaces or an algebraic tensor product of vector spaces depending on the context.

A locally compact quantum group $\GG$ carries a lot of structure. We will be using the Hilbert space $\Ltwo(\GG)$ defined to be the G.N.S.~Hilbert space associated to the right Haar measure on $\GG$. The corresponding modular conjugation will be denoted by $J$. The algebras $\C_0(\GG)$ and $\Linf(\GG)$ are naturally represented on $\Ltwo(\GG)$ and, moreover, the Hilbert space $\Ltwo(\hh{\GG})$ defined by the right Haar measure of the \emph{dual} locally compact quantum group (see \cite[Section 8]{kv}) is canonically isomorphic to $\Ltwo(\GG)$. In particular the von Neumann algebra $\Linf(\hh{\GG})$ can also be considered as algebra of operators on $\Ltwo(\GG)$ and the modular conjugation $\hh{J}$  of the right Haar measure of $\hh{\GG}$ is an antiunitary operator on $\Ltwo(\GG)$. 

The modular conjugations $J$ and $\hh{J}$ implement the \emph{unitary antipodes} $\hh{R}$ and $R$ of $\hh{\GG}$ and $\GG$ respectively in the following way:
\begin{equation}\label{RJ}
\begin{aligned}
R(x)&=\hh{J}x^*\hh{J},&\quad{x}\in\Linf(\GG),\\
\hh{R}(y)&=Jy^*J,&\quad{y}\in\Linf(\hh{\GG})
\end{aligned}
\end{equation}
(see \cite[Section 5]{kv}).

Another object which will play an important role in our paper is the Kac-Takesaki operator (or the \emph{right regular representation}) of $\GG$ which is a unitary $W^\GG\in\Linf(\hh{\GG})\vtens\Linf(\GG)\subset\B\bigl(\Ltwo(\GG)\tens\Ltwo(\GG)\bigr)$ (\cite{kv,mnw}) and $\Delta_\GG$ is implemented by $W^\GG$, i.e.~we have
\[
\begin{aligned}
\Delta_\GG(x)&=W^\GG(x\tens\I){W^\GG}^*,&\quad{x}\in\Linf(\GG).
\end{aligned}
\]
The Kac-Takesaki operator $W^{\hh{\GG}}$ corresponding to $\hh{\GG}$ turns out to be equal to $\flip({W^\GG}^*)$, where $\flip$ denotes, here and in what follows, the flip map.

The scaling groups (\cite{kv,mnw}) of $\GG$ and $\hh{\GG}$ will be denoted by $(\tau_t)_{t\in\RR}$ and $(\hh{\tau}_t)_{t\in\RR}$ respectively (or simply $\tau$ and $\hh{\tau}$ for short).

Given a locally compact quantum group $\GG$, the \emph{opposite} and the \emph{commutant} quantum groups $\GG^\op$ and $\GG'$ of $\GG$ can be defined. This was done in \cite[Section 4]{KVvN}, where a thorough discussion of these objects was also included. We will be using this material extensively. It is important to notice that the unitary antipodes of $\GG$ and $\GG^\op$ coincide (which makes sense because the algebras $\Linf(\GG)$ and $\Linf(\GG^\op)$ are equal). By \cite[Proposition 4.2]{KVvN} we have
\[
\hh{\GG^\op}=\hh{\GG}',\quad\hh{\GG'}=\hh{\GG}^\op\quad\text{and}\quad(\GG')^\op=(\GG^\op)'.
\]

Suppose $\GG$ and $\HH$ are locally compact quantum groups and we have an injective normal unital $*$-homomorphism $\gamma\colon\Linf(\HH)\to\Linf(\GG)$ such that $(\gamma\tens\gamma)\comp\Delta_\HH=\Delta_\GG\comp\gamma$. Then $\gamma$ also intertwines the unitary antipodes and scaling groups of $\GG$ and $\HH$ (\cite[Proposition 5.45]{kv}). Also, \cite[Proposition 10.5]{BV} says that if $\sM\subset\Linf(\GG)$ is a von Neumann subalgebra such that $\Delta_\GG(\sM)\subset\sM\vtens\sM$ and $\sM$ is invariant for $R$ and $\tau$ then there exists a locally compact quantum group $\HH$ such that $\sM$ is isomorphic to $\Linf(\HH)$ with a normal $*$-homomorphism intertwining the respective comultiplications.

\subsection{Closed quantum subgroups}\label{cqs}

Let $\GG$ and $\HH$ be locally compact quantum groups. We say that $\HH$ is a \emph{closed quantum subgroup} of $\GG$ in the sense of Vaes, if there exists an injective normal unital $*$-homomorphism $\gamma\colon\Linf(\hh{\HH})\to\Linf(\hh{\GG})$ such that $(\gamma\tens\gamma)\comp\Delta_{\hh{\HH}}=\Delta_{\hh{\GG}}\comp\gamma$. We will use exclusively this notion of a closed quantum subgroup often abbreviating the full name to ``$\HH$ is a subgroup of $\GG$''. A thorough study of closed quantum subgroups is contained in \cite{DKSS}. The results of this paper show that associated to a closed quantum subgroup $\HH$ of $\GG$ there are two actions
\[
\th\colon\Linf(\GG)\longrightarrow\Linf(\GG)\vtens\Linf(\HH)\quad\text{and}\quad\thL\colon\Linf(\GG)\longrightarrow\Linf(\HH)\vtens\Linf(\GG)
\]
which are a \emph{right} and \emph{left quantum group homomorphism} on von Neumann algebra level (\cite{MRW}), i.e.
\[
\begin{split}
(\Delta_\GG\tens\id)\comp\th&=(\id\tens\th)\comp\Delta_\GG,\\
(\id\tens\Delta_\HH)\comp\th&=(\th\tens\id)\comp\th,\\
(\id\tens\Delta_\GG)\comp\thL&=(\thL\tens\id)\comp\Delta_\GG,\\
(\Delta_\HH\tens\id)\comp\thL&=(\id\tens\thL)\comp\thL
\end{split}
\]
(cf.~\cite[Proposition 3.1]{VaesVainerman}).

There are various formulas describing the relationship between $\th$ and $\thL$ (see e.g.~\cite[Lemma 5.7]{MRW}), but the two most important ones for our uses are
\begin{subequations}
\begin{align}
(\id\tens\thL)\comp\Delta_\GG&=(\th\tens\id)\comp\Delta_\GG,\label{DelTh}\\
\hhthL&=\flip\comp(\hh{R}\tens\hh{R})\comp\hh{\th}\comp\hh{R}\label{thetaL}
\end{align}
\end{subequations}
(\cite[Theorem 5.5]{MRW}).

\subsection{Homomorphisms and bicharacters}

We will freely use the (already mentioned) notions of left and right quantum group homomorphisms and their relation to bicharacters as described in \cite{MRW}, cf.~also \cite[Section 2.4]{proj}. 

\section{Co-duality for coideals}\label{coduality}

Let $\GG$ be a locally compact quantum group. A von Neumann subalgebra $\sL\subset\Linf(\GG)$ is a \emph{left coideal} if $\Delta_\GG(\sL)\subset\Linf(\GG)\vtens\sL$. In \cite[Definition 3.6]{KaspSol} for each left coideal $\sL\subset\Linf(\GG)$ a left coideal $\dd{\sL}\subset\Linf(\hh{\GG})$ was defined by
\[
\dd{\sL}=\sL'\cap\Linf(\hh{\GG}).
\]
We call the coideal $\dd{\sL}$ the \emph{co-dual} of $\sL$. By \cite[Theorem 3.9]{KaspSol} we have $\dd{\dd{\sL}}=\sL$ for any left coideal $\sL\subset\Linf(\GG)$. A very similar duality for coideals was developed in \cite{tomatsu} for the case of compact quantum groups.

In the next proposition we use the fact that the unitary antipodes of $\GG$ and $\GG^\op$ coincide (see Section \ref{lcqgs}).

\begin{proposition}\label{JLJ}
Let $\sL\subset\Linf(\GG)$ be a left coideal. Define $\sK=J\sL{J}$. Then $\sK\subset\Linf(\GG')$ and $\sK$ is a left coideal in $\Linf(\GG')$. Moreover the co-dual $\dd{\sK}\subset\Linf(\hh{\GG'})=\Linf(\hh{\GG}^\op)$ satisfies
\[
\dd{\sK}=\hh{R}(\dd{\sL}).
\]
\end{proposition}

\begin{proof}
Take $x\in\sL$ and let $x'=JxJ$. Since $\sL$ is a left coideal, we have
\[
\begin{split}
\Delta_{\GG'}(x')&=(J\tens{J})\Delta_\GG(Jx'J)(J\tens{J})\\
&=(J\tens{J})\Delta_\GG(x)(J\tens{J})\in\Linf(\GG')\vtens\sK.
\end{split}
\]
Furthermore
\[
\begin{split}
\dd{\sK}&=\sK'\cap\Linf(\hh{\GG}^\op)\\
&=J\sL'J\cap\Linf(\hh{\GG}^\op)\\
&=J\bigl(\sL'\cap\Linf(\hh{\GG}^\op)\bigr)J\\
&=\hh{R}(\dd{\sL})
\end{split}
\]
because $\hh{R}$ is implemented by $J$ (Equation \eqref{RJ}).
\end{proof}

Let us note that Proposition \ref{JLJ} in particular says that $\dd{\sL}$ is $\hh{R}$-invariant if and only if
\begin{equation}\label{KL}
\dd{\sK}=\dd{\sL}.
\end{equation}
This however does not mean that $\sK=\sL$ because the co-dualities on both sides of \eqref{KL} are different: $\dd{\sK}$ is the co-dual of a coideal in $\Linf(\GG')$, while $\dd{\sL}$ is a co-dual of a coideal in $\Linf(\GG)$.

The next result deals with invariance of a coideal under the scaling group. We will say that a von Neumann subalgebra $\sN\subset\Linf(\GG)$ is \emph{$\tau$-invariant} if $\tau_t(\sN)=\sN$ for all $t\in\RR$.

\begin{proposition}\label{tauinv}
Let $\sL\subset\Linf(\GG)$ be a left coideal. Then $\sL$ is $\tau$-invariant if and only if $\dd{\sL}$ is $\hh{\tau}$-invariant.
\end{proposition}

\begin{proof}
This is a very simple computation. The scaling groups of $\GG$ and $\hh{\GG}$ are implemented by the same strongly continuous one parameter group of unitaries $(Q^{2it})_{t\in\RR}$ on $\Ltwo(\GG)$ (\cite[Theorem 1.5(5)]{mu}, \cite[Proposition 6.10]{kv}, cf.~also \cite{remmu} and \cite[Proof of Proposition 39(4)]{modmu}). Assume that $\sL$ is $\tau$-invariant. Then for $x\in\sL$, $y\in\dd{\sL}$ and $t\in\RR$ we have
\[
\tau_t(x)y=Q^{2it}xQ^{-2it}y=yQ^{2it}xQ^{-2it},
\]
so  $Q^{-2it}yQ^{2it}x=xQ^{-2it}yQ^{2it}$. Since this holds for any $t\in\RR$ and all $x\in\sL$, $y\in\dd{\sL}$, we conclude that $\dd{\sL}$ is $\hh{\tau}$-invariant. The converse implication is proved analogously.
\end{proof}

Let us now describe a class of coideals which will play a distinguished role in this paper. Let $\GG$ and $\HH$ be locally compact quantum groups and assume that $\HH$ is  closed quantum subgroup of $\GG$ via $\gamma\colon\Linf(\hh{\HH})\hookrightarrow\Linf(\hh{\GG})$. As mentioned in Section \ref{cqs} there are a left and a right quantum group homomorphisms $\th\colon\Linf(\GG)\to\Linf(\GG)\vtens\Linf(\HH)$ and $\thL\colon\Linf(\GG)\to\Linf(\HH)\vtens\Linf(\GG)$ associated with the pair $\HH\subset\GG$. Furthermore the image $\sL$ of $\gamma$ is a left (and also right) coideal in $\Linf(\hh{\GG})$ (in fact $\Delta_{\hh{\GG}}(\sL)\subset\sL\vtens\sL$). We have

\begin{proposition}\label{homog}
With the assumptions above we have
\begin{equation}\label{invth}
\dd{\sL}=\bigl\{x\in\Linf(\GG)\st\th(x)=x\tens\I\bigr\}\quad\text{and}\quad
R\bigl(\dd{\sL}\bigr)=\bigl\{x\in\Linf(\GG)\st\thL(x)=\I\tens{x}\bigr\}.
\end{equation}
\end{proposition}

\begin{proof}
The map $\th$ is implemented by $V=(\gamma\tens\id)W^\HH$ in the following sense:
\[
\begin{aligned}
\th(x)&=V(x\tens\I)V^*,&\quad{x}\in\Linf(\GG)
\end{aligned}
\]
(\cite[Theorem 5.3]{MRW}, \cite[Theorem 3.3]{DKSS}). Since slices of $V$ generate $\sL=\gamma\bigl(\Linf(\hh{\HH})\bigr)$, for $x\in\Linf(\GG)$, to commute with $\sL$ is equivalent to $\th(x)=x\tens\I$. This proves the first formula of \eqref{invth}. The second one is a direct consequence of \eqref{thetaL}. 
\end{proof}

Let $\GG,\HH$ and $\sL$ be as in Proposition \ref{homog}. The coideal $\dd{\sL}\subset\Linf(\GG)$ is called in \cite[Definition 4.1]{impr} \emph{measured quantum homogeneous space}. We will sometimes use the notation $\dd{\sL}=\Linf(\GG/\HH)$.

\section{Normal coideals}\label{normalSect}

\begin{definition}\label{normalDef}
Let $\GG$ be a locally compact quantum group and let $\sL\subset\Linf(\GG)$ be a left coideal. We say that $\sL$ is \emph{normal} if
\[
{W^{\hh{\GG}}}^*(\sL\tens\I)W^{\hh{\GG}}\subset\sL\vtens\Linf(\hh{\GG}).
\]
\end{definition}

\begin{remark}
Equivalently a left coideal $\sL\subset\Linf(\GG)$ is normal if $W^\GG(\I\tens\sL){W^\GG}^*\subset\Linf(\hh{\GG})\vtens\sL$. The definition is a direct extension of the notion of a normal quantum subgroup from \cite{VaesVainerman2}. A similar notion appears also in \cite{tomatsu}, where an analog of Proposition \ref{normalProp} was proved.
\end{remark}

\begin{proposition}\label{normalProp}
Let $\GG$ be a locally compact quantum group and $\sL\subset\Linf(\GG)$ a left coideal. Then $\Delta_\GG(\sL)\subset\sL\vtens\sL$ if and only if $\dd{\sL}$ is normal.
\end{proposition}

\begin{proof}
Assume $\dd{\sL}$ is normal and take $y\in\dd{\sL}$. For $u={W^{\GG}}^*(y\tens\I)W^{\GG}$ and $x\in\sL$ we compute using the fact that $u\in\dd{\sL}\vtens\Linf(\GG)$:
\[
\begin{split}
(y\tens\I)\Delta_\GG(x)&=(y\tens\I)W^\GG(x\tens\I){W^\GG}^*\\
&=W^\GG{W^\GG}^*(y\tens\I)W^\GG(x\tens\I){W^\GG}^*\\
&=W^\GG{u}(x\tens\I){W^\GG}^*\\
&=W^\GG(x\tens\I)u{W^\GG}^*\\
&=W^\GG(x\tens\I){W^\GG}^*W^\GG{u}{W^\GG}^*\\
&=\Delta_\GG(x)(y\tens\I).
\end{split}
\]
As this is true for all $y\in\sL$, we conclude that $\Delta_\GG(x)\in\sL\vtens\Linf(\GG)$, and so $\Delta_\GG(\sL)\subset\sL\vtens\Linf(\GG)$. But $\sL$ is assumed to be a left coideal: $\Delta_\GG(\sL)\subset\Linf(\GG)\vtens\sL$. It follows that $\Delta_\GG(\sL)\subset\sL\vtens\sL$.

Suppose now that $\Delta_\GG(\sL)\subset\sL\vtens\sL$. Take $y\in\dd{\sL}$ and for any $x\in\sL$ compute
\[
\begin{split}
{W^\GG}^*(y\tens\I)W^\GG(x\tens\I)&={W^\GG}^*(y\tens\I)W^\GG(x\tens\I){W^\GG}^*W^\GG\\
&={W^\GG}^*(y\tens\I)\Delta_\GG(x)W^\GG\\
&={W^\GG}^*\Delta_\GG(x)(y\tens\I)W^\GG\\
&={W^\GG}^*\Delta_\GG(x)W^\GG{W^\GG}^*(y\tens\I)W^\GG\\
&=(x\tens\I){W^\GG}^*(y\tens\I)W^\GG.
\end{split}
\]
Since this is true for all $x\in\sL$, we get ${W^\GG}^*(y\tens\I)W^\GG\in\dd{\sL}\vtens\Linf(\GG)$.
\end{proof}

\begin{proposition}\label{RL}
Let $\sL\subset\Linf(\GG)$ be a left coideal such that $R(\sL)=\sL$, then $\dd{\sL}$ is normal. 
\end{proposition}

\begin{proof}
Using the fact that $\sL$ is a left coideal and properties of the unitary antipode we compute
\[
\begin{split}
\Delta_\GG(\sL)&=\Delta_\GG\bigl(R(\sL)\bigr)\\
&=\flip\bigl((R\tens{R})\Delta_\GG(\sL)\bigr)\\
&=\flip\bigl((R\tens{R})\Delta_\GG(\sL)\bigr)
\subset\flip\bigl(R(\Linf(\GG))\vtens{R(\sL)}\bigr)=\sL\vtens\Linf(\GG).
\end{split}
\]
Thus $\Delta_\GG(\sL)\subset\sL\vtens\sL$ and $\sL$ is normal by Proposition \ref{normalProp}. 
\end{proof}

\begin{definition}\label{strongly}
Let $\sL$ be a left coideal in $\Linf(\GG)$. We say that $\sL$ is \emph{strongly normal} if $\sL$ is $\tau$-invariant and its co-dual $\dd{\sL}$ satisfies $\hh{R}(\dd{\sL})=\dd{\sL}$.
\end{definition}

By Proposition \ref{RL} any strongly normal coideal is normal. Furthermore, it follows from \cite[Proposition 10.5]{BV} that $\sL\subset\Linf(\GG)$ is a strongly normal coideal if and only if $\dd{\sL}\subset\Linf(\hh{\GG})$ has the structure of the algebra of functions on a locally compact quantum group (with comultiplication inherited from $\Linf(\hh{\GG})$). In other words strongly normal coideals in $\Linf(\GG)$ are exactly co-duals of subalgebras of $\Linf(\hh{\GG})$ preserved by $\Delta_{\hh{\GG}},\hh{R}$ and $\hh{\tau}$.

It turns out that for coideals of a special form (discussed at the end of Section \ref{coduality}) normality and strong normality are equivalent. The next theorem is basically contained in \cite[Theorem 2.11]{VaesVainerman2}.

\begin{theorem}\label{strongThm}
Let $\GG$ and $\HH$ be locally compact quantum groups and let $\HH$ be a closed quantum subgroup of $\GG$ in the sense of Vaes with $\gamma\colon\Linf(\hh{\HH})\hookrightarrow\Linf(\hh{\GG})$. Let $\sL\subset\Linf(\hh{\GG})$ be the image of $\gamma$. Then $\sL$ is a left coideal and $\sL$ is normal if and only if $\sL$ is strongly normal.
\end{theorem}

\begin{proof}
We have $\Delta_{\hh{\GG}}(\sL)\subset\sL\vtens\sL$, so in particular $\sL$ is a left coideal. Moreover $\sL$ is $\tau$-invariant, so we must show that if $\sL$ is normal then $R(\dd{\sL})=\dd{\sL}$.

As $\dd{\sL}$ is a left coideal in $\Linf(\GG)$, $\Delta_{\GG}$ restricted to $\dd{\sL}$ is an action of $\GG$ on the von Neumann algebra $\dd{\sL}$, by \cite[Corollary 2.7]{proj} we have
\begin{equation}\label{slPod}
\dd{\sL}=\linW\bigl\{(\omega\tens\id)\Delta_\GG(y)\st{y}\in\dd{\sL},\:\omega\in\B(\Ltwo(\GG))_*\bigr\}.
\end{equation}

Let $\th$ and $\thL$ be the right and left quantum group homomorphism associated with the closed quantum subgroup $\HH$ of $\GG$ (cf.~Proposition \ref{homog} and preceding remarks). Then we have $\dd{\sL}=\bigl\{x\in\Linf(\GG)\st\th(x)=x\tens\I\bigr\}$ and $R\bigl(\dd{\sL}\bigr)=\bigl\{x\in\Linf(\GG)\st\thL(x)=\I\tens{x}\bigr\}$.

The assumption that $\sL$ is normal means that $\Delta_\GG(\dd{\sL})\subset\dd{\sL}\vtens\dd{\sL}$. Thus by \eqref{slPod} we have
\begin{equation}\label{th1}
\linW\bigl\{(\omega\tens\id\tens\id)((\th\tens\id)\Delta_\GG(x))|x\in\dd{\sL},\:\omega\in\B(\Ltwo(\GG))_*\bigr\}=\I\tens\dd{\sL}.
\end{equation}
On the other hand, due to \eqref{DelTh}, the left hand side of \eqref{th1} is
\[
\linW\bigl\{(\omega\tens\id\tens\id)((\id\tens\thL)\Delta_\GG(x))|x\in\dd{\sL},\:\omega\in\B(\Ltwo(\GG))_*\bigr\}
\]
which by \eqref{slPod} is equal to $\thL\bigl(\dd{\sL}\bigr)$. Therefore $R(\dd{\sL})\subset\dd{\sL}$. The unitary antipode is involutive, so $R(\dd{\sL})=\dd{\sL}$.
\end{proof}

\begin{remark}\label{strongRem}
In the situation from Theorem \ref{strongThm} we easily find that the following conditions are equivalent
\begin{enumerate}
\item $\sL$ is normal,
\item $\sL$ is strongly normal,
\item For any $y\in\dd{\sL}$ we have $\thL(x)=x\tens\I$.
\end{enumerate}
These statements are also equivalent to normality of the quantum subgroup $\HH$ (cf.~\cite{VaesVainerman2}).
\end{remark}

\section{Quantum groups with projection}\label{projSect}

Let us recall the definition of a locally compact quantum group with projection from \cite{proj}.

\begin{definition}\label{QGproj}
Let $\GG$ and $\HH$ be locally compact quantum groups. We say that $\GG$ is a \emph{quantum group with projection onto $\HH$} if there exists a unital normal $*$-homomorphism $\pi\colon\Linf(\HH)\to\Linf(\GG)$ and a right quantum group homomorphism $\th\colon\Linf(\GG)\to\Linf(\GG)\vtens\Linf(\HH)$ (see Section \ref{cqs}) such that
\[
\begin{split}
(\pi\tens\pi)\comp\Delta_\HH&=\Delta_\GG\comp\pi,\\
(\pi\tens\id)\comp\Delta_\HH&=\th\comp\pi.
\end{split}
\]
\end{definition}

The results of \cite[Section 3]{proj} show, in particular, that $\pi$ is injective, so considering $\Linf(\HH)$ as embedded into $\Linf(\GG)$ we have that $\Delta_\HH=\bigl.\Delta_\GG\bigr|_{\Linf(\HH)}$. 

Let us give a short recapitulation of the main results of \cite{proj}. It turns out that we also have $\Linf(\hh{\HH})\subset\Linf(\hh{\GG})$ and $\Delta_{\hh{\HH}}=\bigl.\Delta_{\hh{\GG}}\bigr|_{\Linf(\hh{\HH})}$. In particular all algebras are naturally realized as von Neumann algebras acting on $\Ltwo(\GG)$. Moreover, the bicharacter $U$ (\cite{MRW}) corresponding to the right quantum group homomorphism $\th$ can be identified with the reduced bicharacter (the multiplicative unitary, cf.~discussion in \cite{modmu}) of $\HH$. Thus we have two multiplicative unitaries $W^\GG$ and $U$ on $\Ltwo(\GG)$ and the latter is a bicharacter. 

Defining $F$ by
\begin{equation}\label{WFU}
W^\GG=FU,
\end{equation}
i.e.~$F=W^\GG{U^*}$ we obtain a unitary operator acting on $\Ltwo(\GG)\tens\Ltwo(\GG)$ and the ultra-weak closure $\sN$ of the set
\[
\bigl\{(\omega\tens\id)F\st\omega\in\B(\Ltwo(\GG))_*\bigr\}
\]
is a von Neumann subalgebra of $\Linf(\GG)$ which coincides with
\[
\bigl\{x\in\Linf(\GG)\st\th(x)=x\tens\I\bigr\}.
\]
We call $\sN$ the \emph{fixed point subalgebra} corresponding the $\GG$ and $\HH$.

The algebra $\Linf(\GG)$ is shown to be isomorphic to a crossed product of $\sN$ by an action of $\hh{\HH}^\op$ in such a way that $\th$ can be identified with the dual action of $\HH$ (after an appropriate isomorphism, cf.~\cite[Proposition 4.1]{proj}). Furthermore the comultiplication $\Delta_\GG$ restricted to $\sN$ becomes a map $\sN\to\sN\boxtimes\sN$, where the latter algebra is defined to be the von Neumann algebra generated by two copies of $\sN$ inside $\Linf(\GG)\vtens\Linf(\GG)$, namely $\sN\tens\I$ and $\thL(\sN)$, where $\thL\colon\Linf(\GG)\to\Linf(\HH)\vtens\Linf(\GG)$ is a left quantum group homomorphism related to $\th$ by
\[
(\id\tens\thL)\comp\Delta_\GG=(\th\tens\id)\comp\Delta_\GG
\]
(this equality determines $\thL$, cf.~Section \ref{cqs}). The algebra $\sN\boxtimes\sN$ was called the \emph{braided tensor product} of $\sN$ with itself in \cite{proj}.

Let us note here that \cite[Proof of Theorem 4.7]{proj} shows, in particular, that
\begin{equation}\label{proof47}
(\id\tens\Delta_\GG)F=F_{12}\bigl((\id\tens\thL)F\bigr).
\end{equation}

In Section \ref{ExtSect} we will need the following fact.

\begin{lemma}\label{lemDDD}
We have
\begin{equation}\label{DDD}
F_{23}F_{12}F_{23}^*=F_{12}\bigl((\id\tens\thL)F\bigr).
\end{equation}
\end{lemma}

\begin{proof}
The unitary $U\in\Linf(\hh{\GG})\vtens\Linf(\HH)$ is a bicharacter, so in particular $(\Delta_{\hh{\GG}}\tens\id)U=U_{23}U_{13}$. Applying $\flip\tens\id$ to both sides of this equation we obtain
\begin{equation}\label{WFW}
{W^{\GG}_{12}}^*U_{23}W^{\GG}_{12}=U_{13}U_{23}.
\end{equation}

By \eqref{proof47} the right hand side of \eqref{DDD} is equal to $(\id\tens\Delta_\GG)F$ which is $W^\GG_{23}F_{12}{W^\GG_{23}}^*$. Using \eqref{WFU}, \eqref{WFW} and the pentagon equation for $U$ we compute
\[
\begin{split}
W^\GG_{23}F_{12}{W^\GG_{23}}^*&=W^\GG_{23}W^\GG_{12}U_{12}^*{W^\GG_{23}}^*\\
&=W^\GG_{23}W^\GG_{12}(U_{23}^*U_{13}^*)U_{13}U_{23}U_{12}^*{W^\GG_{23}}^*\\
&=W^\GG_{23}W^\GG_{12}({W^\GG_{12}}^*U_{23}^*W^\GG_{12})U_{13}U_{23}U_{12}^*{W^\GG_{23}}^*\\
&=W^\GG_{23}U_{23}^*W^\GG_{12}(U_{13}U_{23}U_{12}^*){W^\GG_{23}}^*\\
&=W^\GG_{23}U_{23}^*W^\GG_{12}(U_{12}^*U_{23}){W^\GG_{23}}^*=F_{23}F_{12}F_{23}^*
\end{split}
\]
which proves \eqref{DDD}.
\end{proof}

\begin{remark}\label{endsect}
As mentioned in \cite[Remark 3.6]{proj} the dual $\hh{\GG}$ of $\GG$ automatically is a locally compact quantum group with projection onto $\hh{\HH}$. Thus there are corresponding maps $\hh{\pi}\colon\Linf(\hh{\HH})\to\Linf(\hh{\GG})$ and $\hh{\th}\colon\Linf(\hh{\GG})\to\Linf(\hh{\GG})\vtens\Linf(\hh{\HH})$ as well as the left version $\hhthL\colon\Linf(\hh{\GG})\to\Linf(\hh{\HH})\vtens\Linf(\hh{\GG})$. As $\hh{\th}$ is implemented by $\hh{U}$ (cf.~\cite[Remark 3.6]{proj}), the fixed point subalgebra
\[
\hh{\sN}=\bigl\{y\in\Linf(\hh{\GG})\st\hh{\th}(y)=y\tens\I\bigr\}
\]
can be identified with the co-dual $\dd{\Linf(\HH)}$ of $\Linf(\HH)$ (cf.~end of Section \ref{coduality}).
\end{remark}

\section{Extensions of quantum groups}\label{ExtSect}

We begin this section with the definition of an extension of locally compact quantum groups from \cite{VaesPhd,VaesVainerman}. Let us first recall the definition of a measured quantum homogeneous space already mentioned at the end of Section \ref{coduality}. If $\GG$ and $\GG_1$ are locally compact quantum groups and $\GG_1$ is a closed quantum subgroup of $\GG$ with an embedding $\beta\colon\Linf(\hh{\GG_1})\hookrightarrow\Linf(\hh{\GG})$ then the measured quantum homogeneous space $\GG/\GG_1$ is defined by setting $\Linf(\GG/\GG_1)$ to be the co-dual of the coideal $\beta\bigl(\Linf(\hh{\GG_1})\bigr)\subset\Linf(\hh{\GG})$.

\begin{definition}\label{exten}
Let $\GG,\GG_1$ and $\GG_2$ be locally compact quantum groups. We say that $\GG$ is an \emph{extension of $\GG_2$ by $\GG_1$} if there exist normal, unital and injective $*$-homomorphisms
\[
\alpha\colon\Linf(\GG_2)\longrightarrow\Linf(\GG)\quad\text{and}\quad\beta\colon\Linf(\hh{\GG_1})\longrightarrow\Linf(\hh{\GG})
\]
commuting with respective comultiplications and $\alpha\bigl(\Linf(\GG_2)\bigr)=\Linf(\GG/\GG_1)$.
\end{definition}

The actual definition (\cite[Definition 3.5.2]{VaesPhd}, \cite[Definition 3.2]{VaesVainerman}) is formulated slightly differently from the one given above. First of all $\GG_1$ is replaced by $\hh{\GG_1}$, i.e.~$\beta$ maps $\Linf(\GG_1)$ into $\Linf(\hh{\GG})$ not $\Linf(\hh{\GG_1})$ into $\Linf(\hh{\GG})$. Secondly, in \cite{VaesPhd,VaesVainerman} the condition that $\alpha\bigl(\Linf(\GG_2)\bigr)=\Linf(\GG/\GG_1)$ (which is there rather that $\alpha\bigl(\Linf(\GG_2)\bigr)=\Linf(\GG/\hh{\GG_1})$) is formulated by demanding that $\alpha\bigl(\Linf(\GG_2)\bigr)$ be equal to the fixed point subalgebra for an action of $\hh{\GG_1}$ on $\GG$ related to the right quantum group homomorphism associated with the fact that $\hh{\GG_1}$ is a closed quantum subgroup of $\GG$ via $\beta$. This action is implemented by a unitary operator and the property of being fixed under this action translates into commutation with slices of this operator. Furthermore the paper \cite{VaesVainerman} and thesis \cite{VaesPhd} use left Haar measures to construct $\Ltwo(\GG)$ which results in all algebras of functions on duals being commutants of those that we use. Taking all these points into consideration we arrive at the formulation in Definition \ref{exten}. Thus our modification consists of rewriting the original definition in terms of co-duality while using right Haar measures and replacing $\hh{\GG_1}$ by $\GG_1$. The latter change is introduced in order to be consistent with the classical definition of an extension. Indeed, if $G,G_1$ and $G_2$ are locally compact groups then $G$ is an extension of $G_2$ by $G_1$ if and only if we have $\alpha\colon\Linf(G_2)\to\Linf(G)$ and $\beta\colon\Linf(\hh{G_1})\to\Linf(\hh{\GG})$ as in Definition \ref{exten} ($\beta$ is an inclusion of group von Neumann algebras corresponding to $G_1$ being a closed subgroup of $G$, while $\alpha$ maps functions on the quotient group $G_2=G/G_1$ to functions on $G$ constant on cosets). Contrary to that, the definitions in \cite{VaesPhd,VaesVainerman} would classically mean that $G$ is an extension of $G_2$ by $\hh{G_1}$.

Note that the demand
\[
\alpha\bigl(\Linf(\GG_2)\bigr)=\Linf(\GG/\GG_1)=\dd{\beta\bigl(\Linf(\GG_1)\bigr)}
\]
of  Definition \ref{exten} may be equivalently phrased as
\[
\dd{\alpha\bigl(\Linf(\GG_2)\bigr)}=\beta\bigl(\Linf(\GG_1)\bigr).
\]
It follows that if $\GG$ is an extension of $\GG_2$ by $\GG_1$ then $\GG_2$ is determined (up to isomorphism) by $\GG,\GG_1$ and $\beta$ and similarly $\GG_1$ is determined by $\GG,\GG_2$ and $\alpha$. In particular, given $\GG,\GG_2$ and $\alpha\colon\Linf(\GG_2)\hookrightarrow\Linf(\GG)$ it makes sense to ask if $\GG$ is \emph{an extension of $\GG_2$}. This is precisely the question we want to address in the situation when $\GG$ is a quantum group with projection.

Let $\GG$ be a locally compact quantum group with projection onto $\HH$ with $\pi\colon\Linf(\HH)\to\Linf(\GG)$ and $\th\colon\Linf(\GG)\to\Linf(\GG)\vtens\Linf(\HH)$ as in Definition \ref{QGproj}. By putting $\GG_2=\HH$ and $\alpha=\pi$ we obtain a part of the structure described in Definition \ref{exten}. The obvious facts are contained in the following theorem.

\begin{theorem}
Let $\GG$ be a locally compact quantum group with projection onto $\HH$. Then the following are equivalent
\begin{enumerate}
\item\label{extThm11} $\GG$ is an extension of $\HH$,
\item\label{extThm12} the coideal $\Linf(\HH)\subset\Linf(\GG)$ is strongly normal,
\item the coideal $\dd{\Linf(\HH)}\subset\Linf(\hh{\GG})$ is $\hh{R}$-invariant.
\end{enumerate}
\end{theorem}

\begin{proof}
$\GG$ is an extension of $\HH$ if and only if the co-dual of $\Linf(\HH)\subset\Linf(\GG)$ is (the image under $\beta$ of) $\Linf(\GG_1)$ for some locally compact quantum group $\GG_1$. Thus \eqref{extThm11} is equivalent to \eqref{extThm12} (cf.~Definition \ref{strongly}). Note now that $\Linf(\HH)\subset\Linf(\GG)$ is $\tau$-invariant (regardless of whether $\GG$ is an extension of $\HH$), so that $\dd{\Linf(\HH)}\subset\Linf(\hh{\GG})$ is $\hh{\tau}$-invariant by Proposition \ref{tauinv}. Thus $\Linf(\HH)\subset\Linf(\GG)$ is strongly normal if and only if $\dd{\Linf(\HH)}\subset\Linf(\hh{\GG})$ is $\hh{R}$-invariant (cf.~Proposition \ref{RL}).
\end{proof}

In the proof of the next theorem we will consistently use the widely adopted convention of quantum group theory that for an element $X$ of a tensor product of two von Neumann algebras, say, $\sA\vtens\sB$ the symbol $\hh{X}$ denotes $\flip(X^*)\in\sB\vtens\sA$. 

\begin{theorem}\label{main}
Let $\GG$ be a locally compact quantum group with projection onto $\HH$. Denote by $\sL$ the copy $\pi\bigl(\Linf(\HH)\bigr)$ inside $\Linf(\GG)$ and let $\sN\subset\Linf(\GG)$ be the fixed point subalgebra for $\th$. Then the following are equivalent:
\begin{enumerate}
\item\label{main1} $\sL$ is normal,
\item\label{main2} $\sL$ is strongly normal,
\item\label{main3} for any $y\in\dd{\sL}$ we have $\hhthL(y)=\I\tens{y}$,
\item\label{main4} $\sL$ commutes with $\sN$.
\end{enumerate}
\end{theorem}

\begin{proof}
First let us gather some necessary information. As explained in Section \ref{projSect}, $\hh{\GG}$ is a quantum group with projection onto $\hh{\HH}$ with corresponding maps $\hh{\pi}\colon\Linf(\hh{\HH})\to\Linf(\hh{\GG})$ and $\hh{\th}\colon\Linf(\hh{\GG})\to\Linf(\hh{\GG})\vtens\Linf(\hh{\HH})$ (the left version of $\hh{\th}$ was already used in the formulation of condition \eqref{main3}). The Kac-Takesaki operator $W^{\hh{\GG}}$ has the decomposition
\[
W^{\hh{\GG}}=\hh{D}\hh{U}
\]
(for some unitary $\hh{D}$) which is analogous to \eqref{WFU} for $\hh{\GG}$. By the dual version of Lemma \ref{lemDDD}
\begin{equation}\label{dualDDD}
\hh{D}_{23}\hh{D}_{12}{\hh{D}_{23}}^*=\hh{D}_{12}\bigl((\id\tens\hhthL)\hh{D}\bigr).
\end{equation}

Note now that
\[
W^\GG=\flip\bigl({W^{\hh{\GG}}}^*\bigr)=UD,
\]
where $D=\flip(\hh{D}^*)$, and by \eqref{WFU} we obtain $D=U^*FU$, or in other words,
\begin{equation}\label{WDU}
D=U^*W^\GG.
\end{equation}
In particular, as the unitary antipodes of $\hh{\GG}$ and $\GG$ restricted to $\Linf(\hh{\HH})$ and $\Linf(\HH)$ respectively coincide with unitary antipodes of $\hh{\HH}$ and $\HH$, we have
\[
(\hh{R}\tens{R})F=(\hh{R}\tens{R})(W^\GG{U^*})=\bigl((\hh{R}\tens{R})U\bigr)^*\bigl((\hh{R}\tens{R})W^\GG\bigr)=
U^*W^\GG=D
\]
(cf.~e.g.~\cite[Lemma 40]{modmu}). This means that
\[
\bigl\{(\omega\tens\id)D\st\omega\in\B(\Ltwo(\GG))_*\bigr\}
\]
is an ultra-weakly dense subspace of $R(\sN)$.

Finally, using pentagon equations for $W^\GG$ and $U$ and the fact that $U$ is a bicharacter, we find that
\[
\begin{split}
U^*_{23}W_{23}U_{12}^*W_{12}W_{23}^*U_{23}&=U^*_{23}W_{23}U_{12}^*W_{23}^*W_{23}W_{12}W_{23}^*U_{23}\\
&=U^*_{23}U_{13}^*U_{12}^*W_{12}W_{13}U_{23}\\
&=U_{12}^*U^*_{23}W_{12}W_{13}U_{23}\\
&=U_{12}^*W_{12}W_{12}^*U^*_{23}W_{12}W_{13}U_{23}\\
&=U_{12}^*W_{12}U_{23}^*U_{13}^*W_{13}U_{23}
\end{split}
\]
which by \eqref{WDU} means that
\begin{equation}\label{DU}
D_{23}D_{12}D_{23}^*=D_{12}U^*_{23}D_{13}U_{23}.
\end{equation}

Returning to the proof of our theorem we first note that by Theorem \ref{strongThm} and Remark \ref{strongRem} statements \eqref{main1} -- \eqref{main3} are equivalent. 

We proceed now to show \eqref{main3} $\Leftrightarrow$ \eqref{main4}. Equation \eqref{dualDDD} shows that elements of $\hh{\sN}$ are $\hhthL$-invariant if and only if $\hh{D}$ is a multiplicative unitary. This is equivalent to $D$ being a multiplicative unitary and this, by \eqref{DU}, is equivalent to
\begin{equation}\label{above}
U_{23}^*D_{13}U_{23}=D_{13}.
\end{equation}
Now, as the right leg of $D$ generates $R(\sN)$, \eqref{above} is equivalent to the fact that $R(\sN)$ and $\sL$ commute. Since $\sL$ is $R$-invariant, this is equivalent to $\sN$ and $\sL$ commuting, i.e.~to \eqref{main4}.
\end{proof}

\section{Examples}\label{examplesSect}

In this section we will analyze several examples of locally compact quantum groups with projection and show that in very many cases they are not extensions. 

\subsection{Duals of classical semidirect products}

Before dealing with the specific situation of duals of semidirect products let us briefly analyze the behavior of extensions and of quantum groups with projection under passage to the dual quantum group. 

\begin{itemize}
\item[$\blacktriangleright$] If $\GG$ is an extension of $\GG_2$ by $\GG_1$ with $\alpha\colon\Linf(\GG_2)\to\Linf(\GG)$ and $\beta\colon\Linf(\hh{\GG_1})\to\Linf(\hh{\GG})$ as in Definition \ref{exten} then $\hh{\GG}$ is an extension of $\hh{\GG_1}$ by $\hh{\GG_2}$. The corresponding maps $\hh{\alpha}\colon\Linf(\hh{\GG_1})\to\Linf(\hh{\GG})$ and $\hh{\beta}\colon\Linf(\GG_2)\to\Linf(\GG)$ are $\hh{\alpha}=\beta$ and $\hh{\beta}=\alpha$ (cf.~\cite[Remarks after Definition 3.2]{VaesVainerman} and comments after Definition \ref{exten}).
\item[$\blacktriangleright$] If $\GG$ is a locally compact quantum group with projection onto $\HH$ then, as already explained in Remark \ref{endsect}, $\hh{\GG}$ is a locally compact quantum group with projection onto $\hh{\HH}$.
\end{itemize}

The above two observations indicate that the dual of an extension, built from two locally compact quantum groups, is an extension built from the duals of the original two quantum groups, with their roles reversed. On the other hand if $\GG$ is a locally compact quantum group with projection onto $\HH$ then it may or may not be an extension of $\HH$, but the dual $\hh{\GG}$ should naturally be considered as a candidate of an extension of $\hh{\HH}$. Therefore even if a given locally compact quantum $\GG$ group with projection onto $\HH$ happens to be an extension of $\HH$, one should not expect $\hh{\GG}$ to have a natural structure of an extension of $\hh{\HH}$ arising from the fact that $\hh{\GG}$ is a quantum group with projection onto $\hh{\HH}$.

Let us now consider the particular case when $G$ and $H$ are locally compact groups and $G$ is equipped with a projection onto $H$, i.e.~we have a pair of continuous homomorphisms $\rh\colon{G}\to{H}$ and $\imath\colon{H}\to{G}$ such that $\rh\comp\imath=\id_H$ (cf.~\cite[Sections 1 \& 3]{proj}). One can easily see that in this case $G$ is isomorphic to a semidirect product of $K\rtimes{H}$, where $K$ is the kernel of $\rh$. In particular $G$ is an extension (of $H$ by $K$), i.e.~we have the exact sequence
\[
\xymatrix{\{1\}\ar[r]&K\ar[r]&G\ar[r]^{\rh}&H\ar[r]&\{1\}.}
\]

The dual $\hh{G}$ of $G$ is a locally compact quantum group with projection onto $\hh{H}$ and the next proposition answers the question whether $\hh{G}$ is an extension of $\hh{H}$.

\begin{proposition}\label{classGH}
Let $G$ and $H$ be as above. If $\hh{G}$ is an extension of $\hh{H}$ then $G$ is the direct product $G=K\times{H}$.
\end{proposition}

\begin{proof}
For any classical group $F$ the algebra $\Linf(\hh{F})$ is the group von Neumann algebra of $F$. Also, since $G=K\rtimes{H}$, we have $\Linf(\hh{G})=\Linf(\hh{K})\rtimes{H}$. The results of \cite{proj} show that $\Linf(\hh{H})\subset\Linf(\hh{G})$ coincides with the embedding arising from the crossed product structure of $\Linf(\hh{G})$ and the fixed point subalgebra $\hh{N}\subset\Linf(\hh{G})$ is the image of $\Linf(\hh{K})$ under its respective inclusion into the crossed product. 

If $\hh{G}$ is an extension of $\hh{H}$ then by Theorem \ref{main} $\hh{N}$ commutes with $\Linf(\hh{H})$, i.e.~the action of $H$ on $K$ is trivial. Consequently $G=K\times{H}$.
\end{proof}

Proposition \ref{classGH} shows that any semidirect product group $G$ which is not a direct product provides an example of a quantum group with projection (namely $\hh{G}$) which is not an extension.

\subsection{Quantum $\mathrm{U}_q(2)$}

In this section we will illustrate the results of Section \ref{ExtSect} by showing that the compact quantum group $\mathrm{U}_q(2)$ (\cite{Koelink,MH-R,wysoczanski}) defined for a complex deformation parameter $q$ such that $0<|q|<1$ is an extension of the classical torus $\TT$ in the sense of \cite[Definition 3.2]{VaesVainerman} if and only if $q$ is real.

We begin by recalling the definition of $\mathrm{U}_q(2)$: fix a deformation parameter $q\in\CC\setminus\{0\}$ of absolute value strictly less than $1$ and let $\zeta=q/\overline{q}$. The \cst-algebra $\C\bigl(\mathrm{U}_q(2)\bigr)$ is the universal \cst-algebra generated by elements $\alpha,\gamma$ and $z$ such that $\gamma$ is normal and
\[
\begin{aligned}
\alpha^*\alpha+\gamma^*\gamma&=\I,&\alpha\gamma&=\overline{q}\gamma\alpha,\\
\alpha\alpha^*+|q|^2\gamma^*\gamma&=\I,&z\gamma{z^*}&=\zeta^{-1}\gamma,\\
zz^*=z^*z&=\I,&z\alpha{z^*}&=\alpha.
\end{aligned}
\]
Note that normality of $\gamma$ and the relation $\alpha\gamma=\overline{q}\gamma\alpha$ imply that $\alpha\gamma^*=q\gamma^*\alpha$ (cf.~\cite[Section 3.1]{SO(3)}). The comultiplication $\Delta_{\mathrm{U}_q(2)}\colon\C\bigl(\mathrm{U}_q(2)\bigr)\to\C\bigl(\mathrm{U}_q(2)\bigr)\tens\C\bigl(\mathrm{U}_q(2)\bigr)$ is uniquely determined by
\[
\begin{aligned}
\begin{aligned}
\Delta_{\mathrm{U}_q(2)}(\alpha)&=\alpha\tens\alpha-q\gamma^*z\tens\gamma,\\
\Delta_{\mathrm{U}_q(2)}(\gamma)&=\gamma\tens\alpha+\alpha^*z\tens\gamma.
\end{aligned}
&&\quad\Delta_{\mathrm{U}_q(2)}(z)&=z\tens{z},
\end{aligned}
\]
With the above comultiplication, $\mathrm{U}_q(2)$ is a compact quantum group which, moreover, is co-amenable (\cite{bmt}). In particular $\mathrm{U}_q(2)$ has a faithful Haar state (\cite[Section 2.5.3]{FrSkTo}) and it fits into the framework of locally compact quantum groups.

As described in \cite{kmrw}, the quantum group $\mathrm{U}_q(2)$ is a locally compact quantum group with projection: consider the mapping $\varLambda\colon\C\bigl(\mathrm{U}_q(2)\bigr)\to\C\bigl(\mathrm{U}_q(2)\bigr)$ given by
\[
\varLambda(\alpha)=\I,\quad\varLambda(\gamma)=0\quad\text{and}\quad\varLambda(z)=z.
\]
Then $\varLambda$ is an idempotent unital $*$-homomorphism commuting with the comultiplication. Using $\varLambda$ one can define on $\mathrm{U}_q(2)$ the structure of a locally compact quantum group with projection onto the classical torus $\TT$: the map $\pi\colon\Linf(\TT)\to\Linf\bigl(\mathrm{U}_q(2)\bigr)$ is
\[
\Linf(\TT)\ni{f}\longmapsto{f}(z)\in\Linf\bigl(\mathrm{U}_q(2)\bigr)
\]
and $\th\colon\Linf\bigl(\mathrm{U}_q(2)\bigr)\to\Linf\bigl(\mathrm{U}_q(2)\bigr)\vtens\Linf(\TT)$ is the von Neumann algebra extension of $(\id\tens\varLambda)\comp\Delta_{\mathrm{U}_q(2)}$.

The ``fixed point subalgebra'' $\sN$ inside $\Linf\bigl(\mathrm{U}_q(2)\bigr)$ is the von Neumann algebra completion of the \cst-subalgebra generated by $\alpha$ and $\gamma$ inside $\C\bigl(\mathrm{U}_q(2)\bigr)$ which is isomorphic to the algebra of continuous functions on the braided quantum group $\mathrm{SU}_q(2)$ as defined in \cite{kmrw}.

\begin{theorem}
The following are equivalent
\begin{enumerate}
\item The quantum group $\mathrm{U}_q(2)$ is an extension of $\TT$,
\item the deformation parameter $q$ is real,
\item $\mathrm{SU}_q(2)$ is a compact quantum group.
\end{enumerate}
\end{theorem}

\begin{proof}
By theorem \ref{main} $\mathrm{U}_q(2)$ is an extension of $\TT$ if and only if $\Linf\bigl(\mathrm{SU}_2(2)\bigr)$ and $\Linf(\TT)$ commute inside $\Linf\bigl(\mathrm{U}_2(2)\bigr)$. This implies that $\zeta=1$, i.e.~$q\in\RR$. This is equivalent to $\mathrm{SU}_q(2)$ being a compact quantum group by results of \cite{kmrw}.
\end{proof}

\subsection{Quantum ``$az+b$'' groups}

The quantum ``$az+b$'' groups were introduced in \cite{azb,nazb}. The construction begins with choosing a deformation parameter $q$ which is a non-zero complex number from a certain set of \emph{admissible} values. These include the interval $]0,1]$, even roots of unity as well as other numbers of absolute value strictly less than $1$ (\cite{nazb}). With each choice of $q$ a (locally compact) multiplicative subgroup $\Gamma_q$ of $\CC\setminus\{0\}$ is fixed.

The \cst-algebra $\C_0(\GG)$ of continuous functions vanishing at infinity on the quantum ``$az+b$'' group $\GG$ is defined to be the crossed product $\C_0\bigl(\Gamma_q\cup\{0\}\bigr)\rtimes\Gamma_q$, with the action of $\Gamma_q$ on $\Gamma_q\cup\{0\}$ given by multiplication of complex numbers. An alternative description of $\C_0(\GG)$ is that it is the universal \cst-algebra generated by two elements $a$ and $b$ affiliated with it (\cite{unbo,gen}) such that
\begin{itemize}
\item[$\blacktriangleright$] $a$ and $b$ are normal,
\item[$\blacktriangleright$] the spectra of $a$ and $b$ are contained in $\Gamma_q\cup\{0\}$,
\item[$\blacktriangleright$] $a$ is invertible (and $a^{-1}$ is affiliated with the \cst-algebra),
\item[$\blacktriangleright$] $ab=q^2ba$ and $a^*b=ba^*$.
\end{itemize}
The comultiplication $\Delta_\GG\in\Mor\bigl(\C_0(\GG),\C_0(\GG)\tens\C_0(\GG)\bigr)$ is defined on generators by
\[
\Delta_\GG(a)=a\tens{a},\quad\text{and}\quad\Delta_\GG(b)=a\tens{b}\,\dot{+}\,b\tens\I,
\]
where $\dot{+}$ denotes the closure of the sum of (unbounded) elements affiliated with $\C_0(\GG)\tens\C_0(\GG)$. The Haar measures on $\GG$ was described in \cite{VDazb,haar} and $\GG$ was found to be a locally compact quantum group (we refer to \cite{unbo} for the definition of a \emph{morphism} of \cst-algebras).

This description of $\GG$ shows clearly that $\GG$ is a classical group if and only if $q=1$. In this case $\Gamma_q=\CC\setminus\{0\}$ and $\GG$ is the classical ``$az+b$'' group. 

Each quantum ``$az+b$ group is a quantum group with projection onto the classical group $\Gamma_q$ (\cite[Example 3.7]{proj}). The associated braided quantum group is given by the commutative \cst-algebra $\C_0\bigl(\Gamma_q\cup\{0\}\bigr)$. Clearly this describes a quantum group if and only if $q=1$ and in this case it is the group $\CC$. 

The question whether $\GG$ is an extension of $\Gamma_q$ is answered in the following theorem:

\begin{theorem}
Let $\GG$ be the quantum ``$az+b$'' group with deformation parameter $q$. Then the following are equivalent:
\begin{enumerate}
\item $\GG$ is an extension of $\Gamma_q$,
\item $q=1$ (and so $\GG$ is a classical group),
\item the associated braided quantum group is a classical locally compact group.
\end{enumerate}
\end{theorem}

\end{document}